\documentclass[12pt]{amsart}

\usepackage{float}
\usepackage{geometry}
\usepackage[all]{xy}
\usepackage{graphicx,youngtab}
\usepackage{amssymb, amsmath}
\usepackage{epstopdf}
\usepackage{mathtools}
\usepackage{xr}
\usepackage{stackrel}
\usepackage{slashbox}
\usepackage{arydshln}

\DeclareMathAlphabet{\mathpzc}{OT1}{pzc}{m}{it}

\newtheorem{thm}{Theorem}[section]
\newtheorem{thmintro}{Theorem}

\newtheorem{lem}[thm]{Lemma}

\newtheorem{claim}[thm]{Claim}

\newtheorem{example}[thm]{Example}
\theoremstyle{definition}

\newtheorem{defn-lem}[thm]{Definition-Lemma}
\theoremstyle{remark}
\newtheorem{rem}[thm]{Remark}
\numberwithin{figure}{section}
\numberwithin{table}{section}
\numberwithin{equation}{section}
\newcommand{\eps}{\varepsilon}
\newcommand{\To}{\longrightarrow}

\newcommand{\trans}{{}^t\!}

\mathchardef\mhyphen="2D

           \newcommand{\eq}[1][r]
   {\ar@<-3pt>@{-}[#1]
    \ar@<-1pt>@{}[#1]|<{}="gauche"
    \ar@<+0pt>@{}[#1]|-{}="milieu"
    \ar@<+1pt>@{}[#1]|>{}="droite"
    \ar@/^2pt/@{-}"gauche";"milieu"
    \ar@/_2pt/@{-}"milieu";"droite"}

\newcommand{\rest}{\mathrm{Rest}}
\newcommand{\Conf}{\mathrm{Conf}}
\newcommand{\conf}{\mathfrak{conf}}
\newcommand{\Diff}{\mathrm{Diff}}
\newcommand{\dS}{\mathrm{dS}}
\newcommand{\AdS}{\mathrm{AdS}}
\renewcommand{\S}{\mathrm{S}} 

\newcommand{\Exterior}{\mathchoice{{\textstyle\bigwedge}}%
    {{\bigwedge}}%
    {{\textstyle\wedge}}%
    {{\scriptstyle\wedge}}}

\newcommand{\R}{\mathbb R}  
\newcommand{\Q}{\mathbb Q}  
\newcommand{\Z}{\mathbb Z}  

\newcommand{\N}{\mathbb N}  
\newcommand{\C}{\mathbb C}  

\newcommand{\rightsetse}[1]{%
\hidewidth\rotatebox[origin=c]{-45}{$\xrightarrow{\kern2em}$}
     \rlap{\raisebox{1ex}
     {$\kern-.8em\scriptstyle #1$}}\hidewidth}
\newcommand{\rightsetsw}[1]{%
\hidewidth\rotatebox[origin=c]{45}{$\xleftarrow{\kern2em}$}
     \rlap{\raisebox{.1ex}
     {$\kern-.8em\scriptstyle #1$}}\hidewidth}
\newcommand{\leftsetsw}[1]{%
\hidewidth
     \llap{\raisebox{1ex}
     {$\scriptstyle #1$\kern-.8em}}
    \rotatebox[origin=c]{45}{$\xleftarrow{\kern2em}$}\hidewidth}
\newcommand{\rightsetnw}[1]{%
\hidewidth\rotatebox[origin=c]{135}{$\xrightarrow{\kern2em}$}
     \rlap{\raisebox{1ex}
     {$\kern-.8em\scriptstyle #1$}}\hidewidth}
\newcommand{\rightsetd}[1]{%
\hidewidth\rotatebox[origin=c]{-90}{$\xrightarrow{\kern2em}$}
     \rlap{{$\scriptstyle #1$}}\hidewidth}
\usepackage[breaklinks]{hyperref}


\begin{document}

\baselineskip=15pt

\title{Conformal symmetry breaking operators for anti-de Sitter spaces}

\author{Toshiyuki Kobayashi}
\author{Toshihisa Kubo}
\author{Michael Pevzner}

\address{T. Kobayashi, Kavli Institute for the Physics and Mathematics of the Universe,
and Graduate School of Mathematical Sciences, 
The University of Tokyo, 3-8-1 Komaba, Meguro, Tokyo, 153-8914 Japan}
\email{toshi@ms.u-tokyo.ac.jp}

\address{T. Kubo, Faculty of Economics, Ryukoku University,
67 Tsukamoto-cho, Fukakusa, Fushimi-ku, Kyoto 612-8577, Japan}
\email{toskubo@econ.ryukoku.ac.jp}

\address{M. Pevzner, Laboratoire de Math\'ematiques de Reims,
Universit\'e de Reims-Champagne-Ardenne,
FR 3399 CNRS, F-51687, Reims, France}
\email{pevzner@univ-reims.fr}

\begin{abstract}
For a pseudo-Riemannian manifold $X$ and a totally geodesic hypersurface $Y$,
we consider the problem of constructing and classifying
all linear differential operators $\mathcal{E}^i(X) \to \mathcal{E}^j(Y)$
between the spaces of differential forms 
that intertwine multiplier representations of the Lie algebra of 
conformal vector fields.
Extending the recent results
in the Riemannian setting by
Kobayashi--Kubo--Pevzner [Lecture Notes in Math.~2170, (2016)],
 we construct such differential operators and give a classification
of them in the pseudo-Riemannian setting where 
both $X$ and $Y$
are of constant sectional curvature,
illustrated by the examples of anti-de Sitter spaces and 
hyperbolic spaces.

\vspace{7pt}

\end{abstract}

 \maketitle

\section{Introduction}\label{sec:1}

Let $X$ be a manifold endowed with a pseudo-Riemannian
metric $g$. A vector field $Z$ on $X$ is called \emph{conformal}
if there exists $\rho(Z,\cdot) \in C^\infty(X)$ (\emph{conformal factor}) such that
\begin{equation*}
L_Z g = \rho(Z,\cdot) g,
\end{equation*}
where $L_Z$ stands for the Lie derivative with respect to the vector field $Z$.
We denote by $\conf(X)$ 
the Lie algebra of conformal vector fields on $X$.

Let $\mathcal{E}^i(X)$ be the space of (complex-valued) smooth $i$-forms on $X$.
We define a family of 
multiplier representations of the Lie algebra
$\conf(X)$ on  
$\mathcal{E}^i(X)$ 
$(0 \leq i \leq \dim X)$
with parameter $u\in\C$
by
\begin{equation}\label{eqn:rep}
\Pi^{(i)}_u(Z)\alpha:=L_Z\alpha + \frac{1}{2}u\rho(Z,\cdot) \alpha 
\quad
\text{for $\alpha \in \mathcal{E}^i(X)$}.
\end{equation}
For simplicity, we write $\mathcal{E}^i(X)_u$ for the representation 
$\Pi^{(i)}_u$ of $\conf(X)$ on $\mathcal{E}^i(X)$.

For a submanifold $Y$ of $X$, 
conformal vector fields along $Y$ form
a subalgebra
\begin{equation*}
\conf(X;Y):=
\{Z \in \conf(X) :
Z_y \in T_yY \; \text{for all $y\in Y$}\}.
\end{equation*}

If the metric tensor $g$ is nondegenerate when restricted to the submanifold $Y$,
then $Y$ carries a pseudo-Riemannian metric $g\vert_Y$ and 
there is a natural Lie algebra homomorphism 
$\conf(X;Y) \to \conf(Y)$, $Z\mapsto Z\vert_Y$. In this case
we compare the representation
$\Pi^{(i)}_u$ of 
the Lie algebra $\conf(X)$ on $\mathcal{E}^i(X)$
with an analogous representation denoted by the lowercase letter
$\pi^{(j)}_{v}$ of 
the Lie algebra $\conf(Y)$ on $\mathcal{E}^j(Y)$
for $u,v\in\C$.
For this, we analyze \emph{conformal symmetry breaking operators},
that is, linear maps $T\colon \mathcal{E}^i(X)\to \mathcal{E}^j(Y)$ satisfying
\begin{equation}\label{eqn:csbo}
\pi^{(j)}_v(Z\vert_Y)\circ T = T \circ \Pi^{(i)}_u(Z)
\quad
\text{for all $Z \in \conf(X;Y)$}.
\end{equation}
Some of such operators are given as differential operators
(\emph{e.g.}\ \cite{FG13, Juhl, KKP16, KOSS15, KP16}), 
and others are integral operators
and their analytic continuation (\emph{e.g.}\ \cite{KS13}).
We denote by
$\mathrm{Diff}_{\conf(X;Y)}(\mathcal E^i(X)_{u},\mathcal E^j(Y)_{v})$
the space of differential operators satisfying \eqref{eqn:csbo}.

In the case $X=Y$ and $i=j=0$,
the Yamabe operator, the Paneitz operator \cite{P08},
which appears in four-dimensional supergravity \cite{FT},
or more generally, the so-called GJMS operators \cite{GJMS} are 
such differential operators. 
Branson and Gover \cite{Branson, BG05} 
extended such operators to differential forms when
$i=j$.
The exterior derivative $d$ and the codifferential
$d^*$ also give examples of such operators for $j=i+1$ and $i-1$,
respectively. Maxwell's equations in a vacuum can be expressed
in terms of conformally covariant operators on 2-forms in the Minkowski space $\R^{1,3}$
(see \cite{KW14} for a bibliography). 
All these classical examples concern the case where $X=Y$. On the other hand,
the more general setting where $X\supsetneqq Y$
is closely related to branching laws of infinite-dimensional representations
(cf.\ ``Stage C" of branching problems in \cite{K15}). 
In recent years, for $(X,Y)=(\S^n,\S^{n-1})$, such operators in the scalar-valued
case $(i=j=0)$ were classified 
by Juhl \cite{Juhl}, see also \cite{FG13, K14, KOSS15} for different approaches.
More generally, such operators have been constructed and 
classified also in the matrix-valued case ($i,j$ arbitrary) by the authors
\cite{KKP16}.
In this paper, we give a variant of \cite{KKP16} by extending the framework as follows:
\begin{alignat*}{3}
&\text{group of conformal diffeomorphisms} &&\Longrightarrow &&\; \;
\text{Lie algebra of conformal vector fields};\\
&\text{homogeneous spaces} && \Longrightarrow &&\;\; \text{locally homogeneous spaces};\\
&\text{Riemannian setting} && \Longrightarrow &&\; \;\text{pseudo-Riemannian setting}.
\end{alignat*}

\vskip 0.05in

Let $\R^{p,q}$ denote the space $\R^{p+q}$ endowed with
the flat pseudo-Riemannian metric:
\begin{equation}\label{eqn:Rpq}
g_{\R^{p,q}}=dx_1^2+\cdots + dx^2_p - dy^2_{p+1} - \cdots - dy^2_{p+q}.
\end{equation}

For $p,q \in \N$, we define a submanifold of $\R^{1+p+q}$ by
\begin{equation}\label{eqn:Spq}
\S^{p,q}:=
\begin{cases}
\{(\omega_0,\omega, \eta) \in \R^{1+p+q} : \omega_0^2 + |\omega|^2 - |\eta|^2 = 1 \} & 
(p>0),\\
\{(\omega_0,\eta) \in \R^{1+q}: \omega_0 > 0,\; \omega_0^2 - |\eta|^2 =1 \} & (p=0).
\end{cases}
\end{equation}
Then, the metric $g_{\R^{1+p,q}}$ on the ambient space $\R^{1+p+q}$ induces
a pseudo-Riemannian structure on the hypersurface $\S^{p,q}$
of signature $(p,q)$ with constant sectional 
curvature $+1$, which is sometimes referred to as the (positively curved) 
\emph{space form} of a pseudo-Riemannian manifold.
We may regard $\S^{p,q}$ also as a pseudo-Riemannian manifold
of signature $(q,p)$ with constant curvature $-1$ by using $-g_{\R^{1+p,q}}$
instead, giving rise to the negatively curved space form.

\begin{example}
[Riemannian and Lorentzian cases]
\begin{alignat*}{6}
&\S^{n,0} &&= \mathrm{S}^n \quad && \textnormal{(sphere)},\quad
&&\S^{0,n} &&= \mathrm{H}^n \quad && \textnormal{(hyperbolic space)},\\
&\S^{n-1, 1} &&= \dS^n \quad && \textnormal{(de Sitter space)}, \quad
&&\S^{1, n-1} &&= \AdS^n \quad && \textnormal{(anti-de Sitter space)}.
\end{alignat*}
\end{example}

In Theorems \ref{thm:1613108}--\ref{thm:dim} below, 
we assume $n=p+q \geq 3$ and consider
\begin{equation}\label{eqn:XY}
(X,Y)=\text{$(\S^{p,q},\S^{p-1,q})$, $(\S^{p,q},\S^{p,q-1})$,
$(\R^{p,q},\R^{p-1,q})$, or $(\R^{p,q}, \R^{p,q-1})$}.
\end{equation}

\begin{example}\label{ex:confXY}
$\conf(X;Y) \simeq \mathfrak{o}(p,q+1)$ if 
$(X,Y)=(\S^{p,q},\S^{p-1,q})$ or $(\R^{p,q},\R^{p-1,q})$.
\end{example}

Theorem \ref{thm:1613108} below addresses the question if any
conformal symmetry breaking operator defined locally can be extended globally.

\begin{thmintro}[automatic continuity]\label{thm:1613108}
Let $V$ be any open set of $X$ such that 
$V\cap Y$ is connected and nonempty.
Suppose $u,v \in \C$.
Then the map taking the restriction to $V$ induces a bijection:
\begin{equation*}
\mathrm{Diff}_{\conf(X;Y)}(\mathcal{E}^i(X)_u,
\mathcal{E}^j(Y)_v)
\stackrel{\sim}{\To}
\mathrm{Diff}_{\conf(V,V\cap Y)}(\mathcal{E}^i(V)_u,
\mathcal{E}^j(V\cap Y)_v).
\end{equation*}
\end{thmintro}

We recall from 
 \cite[Chap.~II]{Segal76} that the
pseudo-Riemannian manifolds $\R^{p,q}$ and $\S^{p,q}$ have
a common conformal compactification:
\begin{equation*}
\xymatrix{
\R^{p,q} \ar@{^{(}->}[rd]& & \S^{p,q}\ar@{_{(}->}[ld]\\
&(\S^{p} \times \S^{q})/\Z_2&
}
\end{equation*}
where $(\S^{p}\times \S^{q})/\Z_2$ denotes the direct product
of $p$- and $q$-spheres equipped with the pseudo-Riemannian metric 
$g_{\S^{p}} \oplus (-g_{\S^{q}})$,
modulo the direct product of antipodal maps, see
also \cite[II, Lem.~6.2 and III, Sect.~2.8]{KO03}.
For $X=\R^{p,q}$ or $\S^{p,q}$, we denote by $\overline{X}$ 
this conformal compactification of $X$.

\begin{thmintro}\label{thm:161456}
(1) \emph{(Automatic continuity to the conformal compactification)}.
Suppose $u,v \in \C$ and $0\leq i \leq n$, $0\leq j \leq n-1$.
Then the map taking the restriction to $X$ is a bijection
\begin{equation*}
\Diff_{\conf(\overline{X};\overline{Y})}(\mathcal{E}^i(\overline{X})_u,
\mathcal{E}^j(\overline{Y})_v)
\stackrel{\sim}{\To}
\Diff_{\conf(X;Y)}(\mathcal{E}^i(X)_u, \mathcal{E}^j(Y)_v).
\end{equation*}
(2) If $n\geq 3$, all these spaces are isomorphic to each other for $(X,Y)$ 
in \eqref{eqn:XY} as far as $(p,q)$ satisfies $p+q=n$.
\end{thmintro}

By Theorems \ref{thm:1613108} and \ref{thm:161456},
we see that all conformal symmetry breaking operators 
given locally in some open sets in the 
pseudo-Riemannian case \eqref{eqn:XY} are derived from the 
Riemannian case (\emph{i.e.}\ $p=0$ or $q=0$). 
We note that our representation \eqref{eqn:rep} is normalized
 in a way that $\Pi^{(i)}_{u}$ coincides with
the differential of the representation $\varpi^{(i)}_{u,\delta}$ $(\delta\in \Z/2\Z)$
of the conformal group $\Conf(X)$ introduced in \cite[(1.1)]{KKP16}.
In particular, we can read from \cite[Thm.~1.1]{KKP16}
and \cite[Thm.~2.10]{KKP16} the dimension of 
$\Diff_{\conf(X;Y)}(\mathcal{E}^i(X)_u, \mathcal{E}^j(Y)_v)$ for any 
$i,j,u,v$. 
For simplicity of exposition, we present a coarse feature as follows.

\begin{thmintro}\label{thm:dim}
Suppose $(X,Y)$ is as in \eqref{eqn:XY}, and $V$ any open set of $X$ 
such that $V\cap Y$ is connected and nonempty.
Let $u,v \in \C$, $0\leq i \leq n$, and $0\leq j \leq n-1$.
\begin{enumerate}
\item[(1)] 
For any $u,v \in \C$ and $0\leq i \leq n$, $0\leq j \leq n-1$,
\begin{equation*}
\dim_\C \mathrm{Diff}_{\conf(V;V\cap Y)}
(\mathcal{E}^i(V)_u, \mathcal{E}^j(V\cap Y)_v) \leq 2.
\end{equation*}
\item[(2)] $\mathrm{Diff}_{\conf(V;V\cap Y)}
(\mathcal{E}^i(V)_u, \mathcal{E}^j(V\cap Y)_v) \neq \{0\}$
only if $u, v, i, j$ satisfy
\begin{equation}\label{eqn:nec}
(v+j)-(u+i) \in \N \quad \textnormal{and} \quad
\big(-1 \leq i-j \leq 2 \;\; \textnormal{or} \; \;
n-2 \leq i+j \leq n+1 \big).
\end{equation}
\end{enumerate}
\end{thmintro}

A precise condition when the equality holds in Theorem \ref{thm:dim} (1) 
will be explained in Section \ref{sec:7} in the case $n=4$.
We shall give explicit formul{\ae} of generators of
$\Diff_{\conf(X;Y)}(\mathcal{E}^i(X)_u,\mathcal{E}^j(Y)_v)$
in Theorem \ref{thm:StageC} in Section \ref{sec:2} for the flat pseudo-Riemannian
manifolds, and in Theorem \ref{thm:D} in Section \ref{sec:3} for positively (or negatively) 
curved space forms. 
These operators (with ``renormalization") and their compositions by the Hodge 
star operators with respect to the pseudo-Riemannian metric exhaust all differential
symmetry breaking operators (Remark \ref{rem:20161014}).
The proof of Theorems \ref{thm:1613108}--\ref{thm:dim}
will be given in Section \ref{sec:5}.
\vskip 0.1in

Notation. $\N=\{0,1,2,\cdots\}$, $\N_+=\{1,2,\cdots\}$.

\vskip 0.2in

\emph{Acknowledgements}:
The first author was partially supported by 
Grant-in-Aid for Scientific Research (A) (25247006), Japan Society for the Promotion of
Science. All three authors were partially supported by CNRS Grant PICS n$^\mathrm{o}$ 7270.

\section{Conformally covariant symmetry breaking operators---flat case}\label{sec:2}

In this section, we give explicit formul{\ae} of conformal symmetry breaking operators
in the flat pseudo-Riemannian case 
$(X,Y)=(\R^{p,q},\R^{p-1,q})$ or $(\R^{p,q},\R^{p,q-1})$. This extends the results in \cite{KKP16} 
that dealt with the Riemannian case
$(X,Y)=(\R^n,\R^{n-1})$.

We note that
the signature of the metric restricted to nondegenerate hyperplanes of $\R^{p,q}$ is either
$(p-1,q)$ or $(p,q-1)$. 
Thus it is convenient to introduce two types of coordinates in $\R^{p+q}$ 
accordingly. We set 
\begin{align*}
\R^{p,q}_+&=\{(y,x) \in \R^{q+p}\} 
\quad \text{with}\quad
-dy^2_1-\cdots - dy^2_q + dx^2_{q+1} + \cdots + dx^2_{p+q},\\
\R^{p,q}_-&=\{(x,y) \in \R^{p+q}\} 
\quad \text{with} \quad
dx^2_{1} + \cdots + dx^2_{p}-dy^2_{p+1}-\cdots - dy^2_{p+q}.
\end{align*}
Then by letting the last coordinate to be zero,
we get hypersurfaces of $\R^{p,q}$ of two types:
\begin{equation*}
\R^{p-1,q}_+ \subset \R^{p,q}_+ \; (p\geq 1), \quad
\R^{p,q-1}_- \subset \R^{p,q}_- \; (q \geq 1).
\end{equation*}

For $\ell \in \N$ and $\mu \in \C$,  
we define a family of differential operators on $\R^{p+q}$ by using the above
coordinates:
\begin{alignat}{2}
(\mathcal{D}^\mu_{\ell})_+\equiv
(\mathcal{D}^\mu_{\ell})_{\R^{p,q}_+}
&:=\sum_{k=0}^{\left[\frac{\ell}{2}\right]}a_k(\mu,\ell)
\left(\sum_{j=1}^q\frac{\partial^2}{\partial y_j^2}-\sum_{j=q+1}^{n-1}
\frac{\partial^2}{\partial x_j^2}\right)^k
\left(\frac{\partial}{\partial x_n}\right)^{\ell-2k}
\; &&\text{on $\R^{p,q}_+$},\nonumber\\
(\mathcal{D}^\mu_{\ell})_-\equiv
(\mathcal{D}^\mu_{\ell})_{\R^{p,q}_-}
&:= \sum_{k=0}^{\left[\frac{\ell}{2}\right]}
a_k(\mu,\ell) 
\left(\sum_{j=1}^p\frac{\partial^2}{\partial x_j^2} 
- \sum_{j=p+1}^{n-1}\frac{\partial^2}{\partial y_j^2}\right)^k
\left(\frac{\partial}{\partial y_n}\right)^{\ell-2k}
\; &&\text{on $\R^{p,q}_-$},\nonumber
\end{alignat}
where we set for $k \in \N$ with $0\leq 2k\leq \ell$
\begin{equation}\label{eqn:ak}
a_k(\mu,\ell)
:=\frac{(-1)^k2^{\ell-2k}\Gamma(\ell-k+\mu)}
{\Gamma(\mu+\left[\frac{\ell+1}{2}\right])k!(\ell-2k)!}.
\end{equation}
In the case $(p,q,\eps) = (n,0,+)$, 
$(\mathcal{D}^{\mu}_\ell)_{\R^{p,q}_\eps}$ coincides with
the differential operator $\mathcal{D}^\mu_\ell$ in 
\cite[(1.2)]{KKP16}, which was originally introduced in \cite{Juhl}
(up to scalar).

The coefficients $a_k(\mu,\ell)$ 
arise from a hypergeometric polynomial
\begin{equation*}
\widetilde{C}^\mu_{\ell}(t)
:=\sum_{k=0}^{\left[\frac{\ell}{2}\right]}a_k(\mu,\ell)t^{\ell-2k}.
\end{equation*}
This is a ``renormalized" Gegenbauer polynomial \cite[II, (11.16)]{KP16}
in the sense that 
$\widetilde{C}^\mu_\ell(t)$ is nonzero for all $\mu \in \C$ and $\ell \in \N$ and
satisfies the Gegenbauer differential equation:
\begin{equation*}
\left((1-t^2)\frac{d^2}{dt^2}-(2\mu+1)t\frac{d}{dt} + \ell(\ell+2\mu)\right)f(t)=0.
\end{equation*}

We set $\mu=: u+ i -\frac{1}{2}(n-1)$ and
$\gamma(\mu,a):=1$ \text{($a$: odd)},
$\mu+\frac{a}{2}$ \text{($a$: even)}.

For parameters $u\in \C$ and $\ell \in \N$, we define
a family of linear operators
\begin{equation*}
(\mathcal{D}^{i\to j}_{u,\ell})_-\colon
\mathcal{E}^i(\R^{p,q}) \to \mathcal{E}^j(\R^{p,q-1})
\end{equation*}
in the coordinates $(x_1, \cdots, x_p, y_{p+1},\cdots, y_{p+q})$ of $\R^{p,q}_-$ 
as follows: For $j=i-1$ or $i$,
\begin{align*}
(\mathcal{D}^{i\to i-1}_{u,\ell})_{\R^{p,q}_-}
&:= \rest_{y_n=0} \circ 
\left((\mathcal{D}^{\mu+1}_{\ell-2})_-\; dd^* \iota_{\frac{\partial}{\partial y_n}}
+ \gamma(\mu,a)(\mathcal{D}^{\mu+1}_{\ell-1})_-\; d^* 
+\frac{u+2i-n}{2}(\mathcal{D}^\mu_\ell)_-\; \iota_{\frac{\partial}{\partial y_n}}\right),\\
(\mathcal{D}^{i\to i}_{u,\ell})_{\R^{p,q}_-}&:=\rest_{y_n=0} \circ
\left(-(\mathcal{D}^{\mu+1}_{\ell-2})_-\; dd^* - \gamma(\mu-\frac{1}{2},\ell)
(\mathcal{D}^\mu_{\ell-1})_-\; d \iota_{\frac{\partial}{\partial y_n}} 
+ \frac{u+\ell}{2}(\mathcal{D}^\mu_{\ell})_-\right).
\end{align*}
Here $d^*\colon \mathcal{E}^i(\R^{p,q}_-) \to \mathcal{E}^{i-1}(\R^{p,q}_-)$
is the codifferential $d^*_{\R^{p,q}_-}= (-1)^i *^{-1} d *$, 
where $*\equiv *_{\R^{p,q}_-}$ is the Hodge operator with respect to 
the pseudo-Riemannian structure on $\R^{p,q}_-$, 
$\iota_{\frac{\partial}{\partial y_n}}$ is the interior multiplication by
the vector field $\frac{\partial}{\partial y_n}$, and 
$(\mathcal{D}^\mu_{\ell})_-$ acts on $\mathcal{E}^i(\R^{p,q}_-)$ 
as a scalar differential operator.

In contrast 
to the case $j=i-1$ or $i$ where
the family of operators $\mathcal{D}^{i\to j}_{u,\ell}$
contains a continuous parameter $u\in\C$ and
discrete one $\ell \in \N$, 
it turns out that
the remaining case where $j \notin\{ i-1, i\}$ 
or its Hodge dual $j\notin \{n-i+1, n-i\}$ 
is not abundant in 
conformal symmetry breaking operators.
Actually, for $j\in \{i-2,i+1\}$,
we define $(\mathcal{D}^{i\to j}_{u,\ell})_{\R^{p,q}_-}$ only for  
special values of $(i,u,\ell)$ as follows:
\begin{alignat*}{3}
&(\mathcal{D}^{i\to i-2}_{n-2i,1})_{\R^{p,q}_-}
&&:= -\rest_{y_n=0} \circ \iota_{\frac{\partial}{\partial y_n}}d^*
&& \quad (2\leq i\leq n-1),\\
&(\mathcal{D}^{n\to n-2}_{1-n-\ell,\ell})_{\R^{p,q}_-}
&&:=-\rest_{y_n=0} \circ \left(\mathcal{D}^{\frac{3-n}{2}-\ell}_{\ell-1}\right)_-
\iota_{\frac{\partial}{\partial y_n}}d^*&& \quad (\ell \in \N_+),\\
&(\mathcal{D}^{i\to i+1}_{0,1})_{\R^{p,q}_-}
&&:= \rest_{y_n=0}\circ d &&\quad (1\leq i \leq n-2),\\
&(\mathcal{D}^{0\to 1}_{1-\ell, \ell})_{\R^{p,q}_-}
&&:=\rest_{y_n=0}\circ \left(\mathcal{D}^{\frac{3-n}{2}-\ell}_{\ell-1}\right)_- d
&& \quad (\ell \in \N_+).
\end{alignat*}

Likewise, for $\R^{p,q}_+$, we define a family of linear operators
\begin{equation*}
(\mathcal{D}^{i \to j}_{u,\ell})_+\colon
\mathcal{E}^i(\R^{p,q}) \to \mathcal{E}^j(\R^{p-1,q})
\end{equation*}
in the coordinates $(y_1, \cdots, y_q, x_{q+1}, \cdots, x_{p+q})$ 
of $\R^{p,q}_+$ with parameters $u \in \C$ and $\ell \in \N$. 
In this case, the formul{\ae} are essentially the same as those in
the Riemannian case $(q=0)$ which were introduced in 
\cite[(1.4)--(1.12)]{KKP16}. 
(The changes from $(\mathcal{D}^{i\to j}_{u,\ell})_{\R^{p,q}_-}$
to $(\mathcal{D}^{i \to j}_{u,\ell})_{\R^{p,q}_+}$ are made by replacing
$y_n =0$ with $x_n=0$, $\frac{\partial}{\partial y_n}$ with $\frac{\partial}{\partial x_n}$,
and $d^*_{\R^{p,q}_-}$ with $-d^*_{\R^{p,q}_+}$.)
For the convenience of the reader, we give formul{\ae} for $j=i-1$ or $i$ 
and omit the case $j=i-2$ and $i+1$.
\begin{align*}
(\mathcal{D}^{i\to i-1}_{u,\ell})_+
&:=\rest_{x_n=0}\circ 
\left(-(\mathcal{D}^{\mu+1}_{\ell-2})_+dd^*
\iota_{\frac{\partial}{\partial x_n}}
-\gamma(\mu,\ell)(\mathcal{D}^{\mu+1}_{\ell-1})_+d^*
+\frac{u+2i-n}{2}(\mathcal{D}^\mu_\ell)_+\iota_{\frac{\partial}{\partial x_n}}\right),\\
(\mathcal{D}^{i \to i}_{u,\ell})_+
&:=\rest_{x_n=0}\circ \left((\mathcal{D}^{\mu+1}_{\ell-2})_+dd^*
-\gamma(\mu-\frac{1}{2},\ell)(\mathcal{D}^\mu_{\ell-1})_+
d\iota_{\frac{\partial}{\partial x_n}}
+\frac{u+\ell}{2}(\mathcal{D}^\mu_\ell)_+\right).
\end{align*}

If $i=j=0$, the operators $(\mathcal{D}^{i\to j}_{u,\ell})_{\R^{p,q}_+}$
reduce to scalar-valued differential operators that are proportional to 
$(\mathcal{D}^\mu_\ell)_+$ because $d^*$ and $\iota_{\frac{\partial}{\partial x_n}}$
are identically zero on $\mathcal{E}^0(X)=C^\infty(X)$.

Theorem \ref{thm:StageC} below gives 
conformal symmetry breaking operators on the flat pseudo-Riemannian manifolds:

\begin{thmintro}\label{thm:StageC}
Let $p+q \geq 3$, 
$0\leq i \leq p+q$, $0\leq j\leq p+q-1$, and $u,v\in\C$. Assume
$j\in\{i-2, i-1, i,i+1\}$ and $\ell:=(v+j)-(u+i) \in \N$. 
(For $j\in\{i-2, i+1\}$, we need an additional condition on 
the quadruple $(i,j,u,v)$, or equivalently, on $(i,j,u,\ell)$
as indicated in the $\R^{p,q}_-$ case.)
Then
\begin{alignat*}{2}
(\mathcal{D}^{i\to j}_{u,\ell})_+ &\in \mathrm{Diff}_{\conf(\R^{p,q};\R^{p-1,q})}
(\mathcal{E}^i(\R^{p,q})_u, \mathcal{E}^j(\R^{p-1,q})_v)
\quad &&\textnormal{for $p\geq1$},\\
(\mathcal{D}^{i\to j}_{u,\ell})_- &\in 
\mathrm{Diff}_{\conf(\R^{p,q};\R^{p,q-1})}
(\mathcal{E}^i(\R^{p,q})_u, \mathcal{E}^j(\R^{p,q-1})_v)
\quad && \textnormal{for $q\geq1$}.
\end{alignat*}
\end{thmintro}

\begin{rem}
In recent years, special cases of Theorem \ref{thm:StageC} have been obtained
as below.
\begin{enumerate}
\item $i=j=0$, $\eps=+$, $q=0$: \cite{Juhl},
see also \cite{FG13, K14, KOSS15} for different approaches.
\item $i=j=0$, $\eps=+$, $p$ and $q$ are arbitrary: \cite[Thm.~4.3]{KOSS15}.
\item $i$ and $j$ are arbitrary, $\eps=+$, $q=0$:
\cite[Thms.~1.5, 1.6, 1.7 and 1.8]{KKP16}.
\end{enumerate}

The main machinery of finding symmetry breaking operators
in various geometric gettings in
\cite{KKP16}, \cite{KOSS15}, and \cite[II]{KP16}
is the ``algebraic Fourier transform of generalized Verma modules"
(\emph{F-method} \cite{K13}), see \cite[I]{KP16} for a detailed 
exposition of the F-method.
\end{rem}

The proof of Theorem \ref{thm:StageC} will be given in Section \ref{sec:6}.

\begin{rem}\label{rem:20161014}
There are a few values
of parameters $(u,\ell,i,j)$ for which $(\mathcal{D}^{i\to j}_{u,\ell})_{\pm}$
vanishes, but we can define nonzero conformal symmetry breaking operators for 
such values by ``renormalization" as in \cite[(1.9), (1.10)]{KKP16}. 
The ``renormalized" operators $(\widetilde{\mathcal{D}}^{i\to j}_{u,\ell})_{\pm}$ and
the compositions $*\circ(\widetilde{\mathcal{D}}^{i\to j}_{u,\ell})_{\pm}$ by the Hodge operator
$*$ for $\R^{p-1,q}$ or $\R^{p,q-1}$
exhaust all conformal differential symmetry breaking operators in our framework,
as is followed from Theorem \ref{thm:161456} (2) and from the classification theorem
\cite[Thms.~1.1 and 2.10]{KKP16} in the Riemannian setting.
\end{rem}

\section{Symmetry breaking operators in the space forms}\label{sec:3}

In this section we explain how to transfer the formul{\ae} for
symmetry breaking operators in the flat case (Theorem \ref{thm:StageC})
to the ones in the space form $\S^{p,q}$ (see Theorem \ref{thm:D}).
In particular, Theorem \ref{thm:D} gives conformal symmetry breaking operators
in the anti-de Sitter space (Example \ref{ex:AdSH}).

We consider the following open dense subsets of the flat space $\R^{p,q}$
and the space form $\S^{p,q}$ (see \eqref{eqn:Spq}), respectively:
\begin{alignat*}{2}
(\R^{p,q}_-)' &:=\{(x,y) \in \R^{p+q} : |x|^2-|y|^2 \neq -4\},&&\\
(\S^{p,q})'&:= \{(\omega_0, \omega, \eta) \in \S^{p,q} : \omega_0 \neq -1\}
&&\subset \R^{1+p+q}.
\end{alignat*}

We define a variant of the stereographic projection and its inverse by
\begin{alignat*}{2}
\Psi\colon (\S^{p,q})'&\To (\R^{p,q}_-)',
&&\quad
(\omega_0,\omega,\eta)\mapsto
\frac{2}{1+\omega_0}(\omega,\eta),\\
\Phi\colon (\R^{p,q}_-)' &\To (\S^{p,q})',
&&\quad
(x,y)\mapsto \frac{1}{|x|^2-|y|^2 + 4}(4-|x|^2+|y|^2, 4x, 4y).
\end{alignat*}

\begin{lem}\label{lem:161428}
The map $\Phi$ is a conformal diffeomorphism from 
$(\R^{p,q}_-)'$ onto $(\S^{p,q})'$
with its inverse $\Psi$, and the conformal factor is given by
\begin{align*}
\Phi^*g_{\S^{p,q}} =\frac{16}{(|x|^2-|y|^2+4)^2}g_{\R^{p,q}_-},
\qquad
\Psi^*g_{\R^{p,q}_-}=\frac{4}{(1+\omega_0)^2}g_{\S^{p,q}}.
\end{align*}
\end{lem}

\begin{proof}
See \cite[I, Lem.~3.3]{KO03}, for instance.
\end{proof}

The pseudo-Riemannian spaces
$\R^{p,q}_+$ and $\R^{p,q}_-$ are obviously isomorphic to each other by
switch of the coordinates
\begin{equation*}
s\colon \R^{p,q}_+ \stackrel{\sim}{\To}\R^{p,q}_-
\quad (y,x) \mapsto (x, y).
\end{equation*}
We set
\begin{equation*}
\Phi_-:=\Phi,
\quad
\Phi_+:=\Phi\circ s,
\quad
\Psi_-:=\Psi,
\quad
\Psi_+:=s\circ \Psi.
\end{equation*}
For $v \in \C$, we define ``twisted pull-back" of differential forms \cite[I, (2.3.2)]{KO03}:
\begin{align}
(\Phi_\pm)_v^*\colon\mathcal{E}^j\left((\R^{p,q}_\pm)'\right)
&\To
\mathcal{E}^j((\S^{p,q})'),
\quad \alpha \mapsto \left(\frac{1+\omega_0}{2}\right)^{-v} \Phi^*\alpha, \label{eqn:tp}\\
(\Psi_\pm)_v^*\colon \mathcal{E}^j((\S^{p,q})') 
&\To
\mathcal{E}^j((\R^{p,q}_\pm)'),
\quad \beta \mapsto \left(\frac{|x|^2-|y|^2+4}{4} \right)^{-v} \Psi^*\beta.
\end{align}
Then $(\Psi_\pm)^*_v$ is the inverse of $(\Phi_\pm)^*_v$
in accordance with $\Psi_\pm = (\Phi_\pm)^{-1}$.

We realize the space forms $\S^{p-1,q}$ $(p \geq 1)$ and $\S^{p, q-1}$ $(q\geq 1)$
as totally geodesic hypersurfaces of $\S^{p,q}$ by letting $\omega_p = 0$ 
and $\eta_q =0$, respectively.
Then $\Phi_\pm$ induce the following diffeomorphisms between hypersurfaces.

\begin{align*}
\xymatrix@R=.5pc{
(\R^{p,q}_-)' \ar[r]^{\sim}_{\Phi_-} & (\S^{p,q})' & (\R^{p,q}_+)' 
\ar[r]^{\sim}_{\Phi_+}& (\S^{p,q})'\\
\cup & \cup & \cup & \cup\\
(\R^{p,q-1}_-)' \ar[r]^{\sim} &(\S^{p,q-1})', & (\R^{p-1,q}_+)' \ar[r]^\sim &(\S^{p-1,q})'
}
\end{align*}

We are ready to transfer the formul{\ae} of conformal symmetry breaking operators
for the flat case (Theorem \ref{thm:StageC}) to those for negatively (or positively)
curved spaces:

\begin{thmintro}\label{thm:D}
For $\eps = \pm$,
let $(\mathcal{D}^{i\to j}_{u,\ell})_{\eps}$ be as in Theorem \ref{thm:StageC}.
Then $(\Phi_{\eps})^*_v\circ (\mathcal{D}^{i\to j}_{u,\ell})_\eps
\circ (\Psi_{\eps})^*_u$, originally
defined in the open dense set $(\S^{p,q})'$ of the space form $\S^{p,q}$, 
extends uniquely
to the whole $\S^{p,q}$ and gives an element in 
$\mathrm{Diff}_{\conf(X;Y)}(\mathcal{E}^i(X)_u, \mathcal{E}^j(Y)_v)$
where
$(X,Y) = (\S^{p,q}, \S^{p,q-1})$ for $\eps =- $ and 
$(X,Y) = (\S^{p,q}, \S^{p-1,q})$ for $\eps= + $.
\end{thmintro}

Here, by a little abuse of notation, we have used the symbol $(\Phi_{\eps})^*_v$ to denote
the operator in the $(n-1)$-dimensional case.

Admitting Theorem \ref{thm:1613108}, we give a proof of Theorem \ref{thm:D}.

\begin{proof}[Proof of Theorem \ref{thm:D}]
Similarly to \cite[Prop.~11.3]{KKP16} in the Riemannian case ($q=0$ and $\eps=+$),
the composition $(\Phi_\eps)^*_v\circ (\mathcal{D}^{i\to j}_{u,\ell})_\eps \circ
(\Phi_\eps)^*_u$ gives an element in $\Diff_{\conf(V;V\cap Y)}(\mathcal{E}^i(V)_u,
\mathcal{E}^j(V\cap Y)_v)$ for $V=(\S^{p,q})'$. Then this operator extends to 
the whole $X=\S^{p,q}$ by Theorem \ref{thm:1613108}.
\end{proof}

The $n$-dimensional anti-de Sitter space $\mathrm{AdS}^n(=\S^{1,n-1})$ contains
the hyperbolic space $\mathrm{H}^{n-1} (=\S^{0,n-1})$ and the anti-de Sitter
space $\mathrm{AdS}^{n-1} (=\S^{1,n-2})$ as totally geodesic hypersurfaces.

\begin{example}[hypersurfaces in the anti-de Sitter space]\label{ex:AdSH}
For $(p,q) = (1,n-1)$, the formul{\ae} in Theorem \ref{thm:D} give conformal
symmetry breaking operators as follows.
\begin{alignat*}{3}
&\mathcal{E}^i(\mathrm{AdS}^n)_u
&&\To \mathcal{E}^j(\mathrm{H}^{n-1})_v 
&&\quad \textnormal{for $\eps = +$},\\
&\mathcal{E}^i(\mathrm{AdS}^n)_u
&&\To \mathcal{E}^j(\mathrm{AdS}^{n-1})_v
&&\quad \textnormal{for $\eps = -$}.
\end{alignat*}
\end{example}

\section{Idea of holomorphic continuation}\label{sec:4}

In this section we explain an idea of holomorphic continuation that 
will bridge between differential symmetry breaking operators in the 
Riemannian setting and those in the non-Riemannian setting.

We begin with an observation
from Example \ref{ex:confXY} that
for any $p,q$ with $p \geq 1$
the Lie algebras
\begin{equation*}
\conf(\S^{p,q};\S^{p-1,q})\simeq \conf(\R^{p,q};\R^{p-1,q})
\simeq \mathfrak{o}(p,q+1)
\end{equation*}
have the same complexification $\mathfrak{o}(n+1,\C)$ as far as $p+q=n$.
In turn to geometry, we shall compare (real) conformal vector fields on 
pseudo-Riemannian manifolds $\S^{p,q}$ or $\R^{p,q}$ of various 
signatures $(p,q)$ via holomorphic 
vector fields on a complex manifold
which contains $\S^{p,q}$ or $\R^{p,q}$ as totally real submanifolds.

Let $X_{\C}$ be a connected complex manifold, and
$\Omega^i(X_{\C})$ the space of holomorphic $i$-forms on $X_{\C}$.
If $X$ is a totally real submanifold, then the restriction map
\begin{equation*}
\rest_X\colon \Omega^i(X_{\C}) \To \mathcal{E}^i(X)
\end{equation*}
is obviously injective.

\begin{defn-lem}\label{deflem:161003}
Suppose $D_\C\colon \Omega^i(X_{\C}) \to \Omega^j(X_{\C})$
is a holomorphic differential operator. Then there is a 
unique differential operator $E \colon \mathcal{E}^i(X) \to \mathcal{E}^j(X)$,
such that 
\begin{equation*}
E\vert_{V\cap X} \circ \rest_{V\cap X}\alpha = \rest_{V\cap X} \circ D_\C\vert_V \alpha
\end{equation*}
for any open set $V$ of $X_{\C}$ with $V\cap X \neq \emptyset$ 
and for any $\alpha \in \Omega^i(V)$.
We say that $D_\C$ is the \emph{holomorphic extension} of $E$.
We write $(\rest_X)_*D_\C$ for $E$.
\end{defn-lem}

If $X$ is a real analytic, pseudo-Riemannian manifold with 
complexification $X_\C$, then a holomorphic analogue of the action \eqref{eqn:rep}
makes sense  by analytic continuation for $Z \in \conf(X)\otimes_\R\C$:
$L_Z$ being understood as the holomorphic Lie derivative with respect to
a holomorphic extension of the vector field $Z$ in 
a complex neighbourhood $U$ of $X$, which acts on
$\alpha \in \Omega^i(U)$; and the conformal factor $\rho(\cdot, \cdot)$
being understood as its holomorphic extension 
(complex linear in the first argument).
Likewise for the pair $X \supset Y$ of pseudo-Riemannian manifolds
 with complexification $X_\C \supset Y_\C$,
 we may consider a holomorphic analogue of 
the covariance condition \eqref{eqn:csbo}. Then we have:

\begin{lem}\label{lem:161463}
Suppose $D_\C \colon \Omega^i(X_\C) \to \Omega^j(Y_\C)$ 
is a holomorphic differential operator, and 
$D=(\rest_{X})_*(D_\C)$. Then $D\colon \mathcal{E}^i(X) \to \mathcal{E}^j(Y)$
satisfies the conformal covariance \eqref{eqn:csbo} if and only if
\begin{equation*}
\pi^{(j)}_v(Z\vert_{Y_\C})\circ D_\C\,\alpha = D_\C \circ \Pi^{(i)}_u(Z)\alpha
\end{equation*}
for any $Z \in \conf(X;Y)\otimes_\R\C$,
any open subset $U$ of $X_\C$ with $U\cap Y_\C \neq \emptyset$
and any $\alpha \in \Omega^i(U)$.
\end{lem}

We define a family of totally real vector spaces of 
$\C^n$ by embedding the space $\R^n = \R^p_x \oplus \R^q_y$ as
\begin{alignat*}{2}
\iota_+&\colon \R^q_y \oplus \R^p_x
\stackrel{\sim}{\To}
\sqrt{-1}\R^q \oplus \R^p
&&=\left\{(\sqrt{-1}y_1, \cdots, \sqrt{-1}y_q, x_{q+1},\cdots x_{p+q}):
x_j, y_j \in \R\right\},\\
\iota_-&\colon \R^p_x \oplus \R^q_y
\stackrel{\sim}{\To}
\R^p \oplus \sqrt{-1}\R^q 
&&= \left\{(x_1, \cdots, x_p, \sqrt{-1}y_{p+1}, \cdots, \sqrt{-1}y_{p+q}):
x_j, y_j \in \R\right\}.
\end{alignat*}

Let us apply Lemma \ref{lem:161463} to the following setting where $n=p+q$.
\begin{alignat*}{7}
&\R^{n}&&\simeq \; &&\R^{p,q}_+
\stackbin[\iota_+]{}{\lhook\joinrel\relbar\joinrel\relbar\joinrel\relbar\joinrel\To}
 &&X_\C=\C^n 
 \stackbin[\iota_-]{}{\longleftarrow\joinrel\relbar\joinrel\relbar\joinrel\relbar\joinrel\rhook} 
 &&
\R^{p,q}_- && \simeq \;&& \R^n\\
& && &&\cup &&\;\cup &&\cup && &&\\
& \R^{n-1}&&\simeq &&\R^{p-1,q}_+ 
\lhook\joinrel\relbar\joinrel\To
&&Y_\C=\C^{n-1} 
\longleftarrow\joinrel\relbar\joinrel\rhook
&&\R^{p,q-1}_- && \simeq &&\R^{n-1}
\end{alignat*}

The holomorphic symmetric 2-tensor
\begin{equation*}
ds^2 = dz_1^2 + \cdots + dz_n^2
\end{equation*}
on $\C^n$ induces a flat pseudo-Riemannian structure on $\R^n$ of
signature $(p,q)$ by restriction via $\iota_{\pm}$. 
The resulting pseudo-Riemannian structures (and coordinates) on $\R^n$ are 
nothing but those of $\R^{p,q}_+$ and $\R^{p,q}_-$ given in Section \ref{sec:2}.

\section{Proof of Theorems \ref{thm:1613108}, \ref{thm:161456}, 
and \ref{thm:dim}}\label{sec:5}

This section gives a proof of Theorems \ref{thm:1613108}, \ref{thm:161456}, 
and \ref{thm:dim}. The key machinery for differential symmetry breaking operators
(SBOs for short) is in threefold:
\begin{enumerate}

\item[(1)] holomorphic extension of differential SBOs (Section \ref{sec:4});

\item[(2)] duality theorem between differential SBOs and homomorphisms for 
generalized Verma modules that encode branching laws
\cite[I, Thm.~2.9]{KP16};

\item[(3)] automatic continuity theorem of differential SBOs
in the Hermitian symmetric setting \cite[I, Thm.~5.3]{KP16}.
\end{enumerate}

We note that both (1) and (2) indicate the independence
of real forms as formulated in Theorem \ref{thm:161456} (2),
whereas (3) appeals to the theory of admissible restrictions of real
reductive groups \cite{K98} for a specific choice of real forms of complex
reductive Lie groups. 
\vskip 0.15in

Let $G$ be $SO_0(p+1,q+1)$, the identity component of the 
indefinite orthogonal group $O(p+1,q+1)$, $P=LN$ a maximal prabolic
subgroup of $G$ with 
Levi subalgebra $\mathrm{Lie}(L)\simeq \mathfrak{so}(p,q)+\R$,
and $H$ the identity component of $P$. Then
$G$ acts conformally on $G/H \simeq \S^{p} \times \S^{q}$
equipped with the pseudo-Riemannian structure $g_{\S^{p}}\oplus (-g_{\S^{q}})$.
Similarly, $H'$ is defined by taking
$G':=SO_0(p,q+1)$ $(\eps = +)$ or $SO_0(p+1,q)$ $(\eps = -)$.

Applying the duality theorem \cite[I, Thm.~2.9]{KP16} to the quadruple
$(G,H,G',H')$, we see that any element in 
\begin{equation}\label{eqn:Verma}
\mathrm{Hom}_{\mathfrak{g}'_\C}
(U(\mathfrak{g}'_\C) \otimes_{U(\mathfrak{p}'_\C)}
(\Exterior^{n-1-j}(\C^{n-1}) \otimes \C_{-v-j}),
U(\mathfrak{g}_\C )\otimes_{U(\mathfrak{p}_\C)}
(\Exterior^{n-i}(\C^n)\otimes \C_{-u-i}))
\end{equation}
with notation as in \cite[Sect.~2.6]{KKP16}
induces a differential symmetry breaking operator 
$\overline{D} \in \Diff_{\conf(\overline{X};\overline{Y})}(
\mathcal{E}^i(\overline{X})_u,\mathcal{E}^j(\overline{Y})_v)$
on the conformal compactification $\overline{X}$, and hence
the one on any open subset $V$ of $X$ with $V\cap Y \neq \emptyset$
by restriction. 
In order to prove Theorem \ref{thm:1613108} and Theorem \ref{thm:161456} (1),
it is sufficient to show the following converse statement.

\begin{claim}\label{claim:AB}
Any $D\in \Diff_{\conf(V;V\cap Y)}(\mathcal{E}^i(V)_u,\mathcal{E}^j(V\cap Y)_v)$
is derived from an element in \eqref{eqn:Verma}.
\end{claim}

Let us prove Claim \ref{claim:AB}.

\vspace{10pt}
\noindent
\textbullet\; Step 1. Reduction to the flat case
\vspace{5pt}

By using the twisted pull-back $(\Phi_\pm)^*_v$ and 
$(\Psi_\pm)^*_v$ (see \eqref{eqn:tp}), we may and do assume that
$X=\R^{p,q}$ $(\simeq \R^n)$ and $Y$ is the hypersurface $\R^{n-1}$ 
given by the condition that 
the last coordinate is zero.
By replacing $V$ with an open subset $V'$ of $\R^n$ with $V\cap \R^{n-1} = V'\cap \R^{n-1}$ 
if necessary, we may further 
assume that $V$ is a convex neighbourhood
of $V\cap \R^{n-1}$ in $\R^n$.

\vspace{10pt}
\noindent
\textbullet\; Step 2. Holomorphic extension
\vspace{5pt}

With the coordinates $x=(x',x_n) \equiv (x_1, \cdots, x_{n-1},x_n)$
of $X=\R^n$, any differential operator 
$D\colon \mathcal{E}^i(\R^n) \to \mathcal{E}^j(\R^{n-1})$ takes
the form 
\begin{equation*}
D = \rest_{x_n=0} \circ \sum_{\alpha \in \N^{n}}a_{\alpha}(x')
\frac{\partial^{|\alpha|}}{\partial x_1^{\alpha_1} \cdots\partial x_n^{\alpha_n}}
\end{equation*}
where $a_{\alpha} \in C^\infty(\R^{n-1}) \otimes 
\mathrm{Hom}_\C\left(\Exterior^i(\C^n),\Exterior^j(\C^{n-1})\right)$
(see \cite[I, Ex.~2.4]{KP16}).
Since $Z_k:=\frac{\partial}{\partial x_k}$ $(1\leq k\leq n-1)$ 
is a Killing vector filed, namely, $Z_k \in \conf(X;Y)$ with $\rho(Z_k,\cdot)\equiv 0$,
the conformal covariance \eqref{eqn:csbo} reduces to 
$L_{\frac{\partial}{\partial x_k}} \circ D = D\circ L_{\frac{\partial}{\partial x_k}}$,
which implies that  the matrix-valued function $a_{\alpha}(x')$ is independent of $x'$
for every $\alpha$. We shall denote $a_\alpha(x')$ simply by $a_{\alpha}$.
Then $D$ extends to a holomorphic differential operator 
$D_\C\colon \Omega^i(\C^n) \to \Omega^j(\C^{n-1})$, by setting
\begin{equation*}
D_\C:=\rest_{z_n=0}\circ \sum_{\alpha \in \N^n}a_{\alpha}
\frac{\partial^{|\alpha|}}{\partial z_1^{\alpha_1} \cdots\partial z_n^{\alpha_n}}.
\end{equation*}

If $D$ satisfies the conformal covariance condition \eqref{eqn:csbo} on
$\mathcal{E}^i(V)$ for all $Z \in \conf(V;V\cap Y)\simeq
\mathfrak{o}(p,q+1)$ or $\mathfrak{o}(p+1,q)$, then
by Lemma \ref{lem:161463},
$D_\C$ satisfies
the holomorphic extension of the condition \eqref{eqn:csbo} on $\Omega^i(\C^n)$
for all $Z\in \conf(V;V\cap Y) \otimes_\R\C  \simeq  \mathfrak{o}(n+1,\C)$.

\vspace{10pt}
\noindent
\textbullet\; Step 3. Automatic continuity in the Hermitian symmetric spaces
$G_\R/K_\R \supset G_\R'/K_\R'$
\vspace{5pt}

Automatic continuity theorem is known for holomorphic differential SBOs 
in the Hermitian symmetric settings \cite[I, Thm.~5.3]{KP16}.
Then our strategy to prove Claim \ref{claim:AB} 
is to utilize the automatic continuity theorem in
the Hermitian symmetric setting by embedding a pair
$(G_\R/K_\R, G'_\R/K'_\R)$ of Hermitian symmetric spaces
into the pair $(\C^n, \C^{n-1})$ of the affine spaces as in Step 2. 
For this, we shall choose a specific real form $G_\R$
of $G_\C:=SO(n+2,\C)$ such that $G_\R$ is the group of 
biholomorphic transformations of 
a bounded symmetric domain in $\C^n$ as below.

Let $Q(\tilde{x}):=-x_0^2+x_1^2+\cdots + x_n^2-x_{n+1}^2$ be the quadratic form on
$\R^{n+2}$, and $G_\R$ the identity component of the isotropy group
\begin{equation*}
\{h \in GL(n+2,\R): Q(h\cdot \tilde{x}) = Q(\tilde{x}) \; 
\text{for all $\tilde{x} \in \R^{n+2}$}\}.
\end{equation*}
Then $K_\R:=G_\R\cap SO(n+2)$
is a maximal compact subgroup of $G_\R \simeq SO_0(n,2)$ such that 
$G_\R/K_\R$ is the Hermitian symmetric space of type IV in the \'E.\ Cartan
classification. We take $G_\R'$ to be the stabilizer of $x_n$. 
Then $G_\R' \simeq SO_0(n-1,2)$.

We use the notation as in \cite[II, Sect.~6]{KP16}, and identify $\C^n$ with the open
Bruhat cell of the complex quadric 
\begin{equation*}
\Q^n\C = \{\tilde{z}\in \C^{n+2}\setminus \{0\}: 
Q(\tilde{z}) = 0\}/ \C^\times
\simeq G_\C/P_\C.
\end{equation*}
Then
$G_\R/K_\R$ is realized as the Lie ball
\begin{equation*}
U=\{z \in \C^n : |z\,\trans z|^2+1-2\bar{z}\, \trans z >0, \; |z\,\trans z| < 1 \}.
\end{equation*}
We compare the real form $G$ of $G_\C$ with Lie algebra
$\mathfrak{conf}(X)\simeq \mathfrak{o}(p+1,q+1)$ in Step 1
and another real form $G_\R \simeq SO_0(n,2)$ in Step 3 $(n=p+q)$.
The point here is that the $G$-orbit $G\cdot o \simeq G/P$ through the origin
$o=eP_\C \in G_\C/P_\C$ is closed in $G_\C/P_\C$,
while the $G_\R$-orbit $G_\R\cdot o\simeq G_\R/K_\R$ is open in $G_\C/P_\C$,
as is summarized in the figure below.
\begin{alignat*}{7}
&G_\R/K_\R&&\simeq \; &&U\; 
\stackbin[\text{open}]{}{\subset} 
 &&\; \C^n 
\qquad\quad\quad\;\; \stackbin[\footnotesize{\text{Bruhat cell}}]{}{\subset} 
 &&
\qquad \quad \; \; \mathbb{Q}^n\C&& \simeq \;\;&& G_\C/P_\C\\
& && && &&\; \cup && \qquad \qquad \cup && &&
\; \; \, \cup\;  \tiny{\text{totally real}}\\
& && &&
&&\, \R^{p,q} \hspace{22pt} 
\stackbin[\footnotesize{\text{conformal compactfication}}]{}{\subset} 
&& \quad  \; \; (S^p\times S^q)/\Z_2&&  \simeq \;\; &&G/P
\end{alignat*}
We note that the $G_\R'$-orbit $G'_\R \cdot o \simeq G_\R'/K_\R'$
is realized as the subsymmetric domain 
$U \cap \{z_n = 0\}\simeq \C^{n-1}$. 
Since the holomorphic differential operator $D_\C$ is defined
on $\Omega^i(\C^n)$,
$D_\C$ induces  a holomorphic differential
operator 
\begin{equation}\label{eqn:DGK}
D_\C\vert_{G_\R/K_\R}\colon
\Omega^i(G_\R/K_\R) \To \Omega^j(G_\R'/K_\R') 
\end{equation}
via the 
inclusion $G_\R/K_\R \simeq U \subset \C^n$.

Then the automatic continuity theorem
\cite[I, Thm.~5.3 (2)]{KP16} (and its proof), 
applied to \eqref{eqn:DGK}
implies that
$D_\C\vert_{G_\R/K_\R}$ is 
derived from an element of \eqref{eqn:Verma}
via the duality theorem in the holomorphic setting
(see \cite[I, Thm.~2.12]{KP16}). Thus the proof of Claim \ref{claim:AB} is completed.
Therefore Theorem \ref{thm:1613108} and Theorem \ref{thm:161456} (1)
follow from \cite[I, Thm.~2.9]{KP16}.

Since \eqref{eqn:Verma} is independent of the choice of real forms, 
Theorem \ref{thm:161456} (2) is now clear.

\begin{proof}[Proof of Theorem \ref{thm:dim}]
Owing to Theorems \ref{thm:1613108} and \ref{thm:161456},
Theorem \ref{thm:dim} is reduced to the Riemannian case
$p=0$ and $\eps = -$ or $q=0$ and $\eps = +$.
Then the assertion follows from the classification results
\cite[Thm.~1.1]{KKP16} for the (disconnected) conformal group and
from a discussion on the connected group case (see \cite[Thm.~2.10]{KKP16}).
\end{proof}

\section{Proof of Theorem \ref{thm:StageC}}\label{sec:6}

In this section, we give a proof of Theorem \ref{thm:StageC} in Section \ref{sec:2} 
by reducing it to the 
Riemannian case $(p,q,\eps) = (n,0,+)$ or $(0,n,-)$
which was established in \cite[Thms.~1.5, 1.6, 1.7 and 1.8]{KKP16}.
For this, we apply Definition-Lemma \ref{deflem:161003} to the totally real 
embedding $\iota_{\pm}\colon \R^{p,q}_\pm \hookrightarrow \C^{p+q}$.

With the coefficients $a_k(\mu,\ell)$ given in \eqref{eqn:ak}, 
we define a family of (scalar-valued) holomorphic differential operators on $\C^n$ by
\begin{equation*}
(\mathcal{D}^\mu_\ell)_\C
:=\sum_{k=0}^{\left[\frac{\ell}{2}\right]}
a_k(\mu,\ell)\left(-\sum_{j=1}^{n-1}\frac{\partial^2}{\partial z_j^2}\right)^k
\left(\frac{\partial}{\partial z_n}\right)^{\ell-2k}, 
\end{equation*}
which are the holomorphic extensions of the operators 
$(\mathcal{D}^\mu_\ell)_{\R^{n,0}_+}$ defined in the Riemannian case,
that is,
$(\rest_{\R^{n,0}_+})_*\big((\mathcal{D}^\mu_\ell)_\C\big) = 
(\mathcal{D}^\mu_\ell)_{\R^{n,0}_+}$.
Likewise, we extend $(\mathcal{D}^{i\to j}_{u,\ell})_{\R^{n,0}_+}$ to 
a (matrix-valued) holomorphic differential operator
\begin{equation*}
(\mathcal{D}^{i\to j}_{u,\ell})_\C\colon 
\Omega^i(\C^n) \To \Omega^j(\C^{n-1})
\end{equation*}
in such a way that $(\rest_{\R^{n,0}_+})_* (\mathcal{D}^{i\to j}_{u,\ell})_\C$
coincides with $(\mathcal{D}^{i\to j}_{u,\ell})_{\R^{n,0}_+}$.
By definition of $(\mathcal{D}^{i\to j}_{u,\ell})_{\R^{p,q}_+}$,
it is readily seen that 
$(\rest_{\R^{p,q}_+})_*(\mathcal{D}^{i\to j}_{u,\ell})_\C
=(\mathcal{D}^{i\to j}_{u,\ell})_{\R^{p,q}_+}$
for all $(p,q)$ with $p+q=n$.
Concerning the other real form $\R^{p,q}_-$, we have the following.

\begin{lem}\label{lem:161482}
$(\rest_{\R^{p,q}_-})_*(\mathcal{D}^{i\to j}_{u,\ell})_\C=
e^{-\frac{\pi \sqrt{-1}(\ell+i-j)}{2}}(\mathcal{D}^{i\to j}_{u,\ell})_{\R^{p,q}_-}$.
\end{lem}

\begin{proof}
The assertion is deduced from the formul{\ae} of $(\rest_{\R^{p,q}_-})_*$ for 
the following basic operators.
(For the convenience of the reader, we also list the cases 
$\R^{p,q}_+$ as well.)
\begin{table}[H]
\begin{center}
\begin{tabular}{c|ccccc}
 & $d_\C$ & $d_\C^*$& $\frac{\partial}{\partial z_n}$ & $\iota_{\frac{\partial}{\partial z_n}}$ &
 $(\mathcal{D}^\mu_\ell)_\C$\\[5pt]
\hline
\rule{0pt}{3ex}
$(\rest_{\R^{p,q}_+})_*$ &
$d_{\R^{p,q}_+}$ &$d^*_{\R^{p,q}_+}$&$\frac{\partial}{\partial x_n}$
& $\iota_{\frac{\partial}{\partial x_n}}$ & $(\mathcal{D}^\mu_{\ell})_{\R^{p,q}_+}$\\[5pt]
$(\rest_{\R^{p,q}_-})_*$ &
$d_{\R^{p,q}_-}$ &$d^*_{\R^{p,q}_-}$&$\frac{1}{\sqrt{-1}} \frac{\partial}{\partial y_n}$
&$\frac{1}{\sqrt{-1}}\iota_{\frac{\partial}{\partial y_n}}$ & 
$e^{-\frac{\pi \sqrt{-1}\ell}{2}}(\mathcal{D}^\mu_\ell)_{\R^{p,q}_-}$
\end{tabular}
\end{center}
\end{table}
\end{proof}

We are ready to complete the proof of Theorem \ref{thm:StageC}.

\begin{proof}[Proof of Theorem \ref{thm:StageC}]
Since $(\mathcal{D}^{i\to j}_{u,\ell})_{\R^{n,0}_+} \in 
\Diff_{\conf(\R^n;\R^{n-1})}(\mathcal{E}^i(\R^n)_u, \mathcal{E}^j(\R^{n-1})_v)$ 
by \cite[Thms.~1.5, 1.6, 1.7, 1.8]{KKP16}, the holomorphic differential operator
$(\mathcal{D}^{i\to j}_{u,\ell})_\C$ satisfies the holomorphic covariance condition
by Lemma \ref{lem:161463}. In turn, we conclude Theorem \ref{thm:StageC}
by Lemmas \ref{lem:161482} and \ref{lem:161463}.

\end{proof}

\section{Four-dimensional example}\label{sec:7}

In contrast to the multiplicity-free theorem (\cite[Thm.~1.1]{KKP16}) for 
differential SBOs for (disconnected) conformal
\emph{groups} 
$(\Conf(X), \Conf(X;Y))$ when $(X,Y) = (\S^n,\S^{n-1})$ $(n \geq 3)$,
it may happen that an analogous statement for 
the \emph{Lie algebras} $(\conf(X), \conf(X;Y))$
does not hold anymore. In fact, for some $u,v,i,j$, one has
\begin{equation}\label{eqn:two}
\dim_\C \Diff_{\conf(X;Y)}(\mathcal{E}^i(X)_u,\mathcal{E}^j(Y)_v) > 1
\end{equation}
(or equivalently, $=2$).

In this section we first address the question when and how \eqref{eqn:two} happens 
and then describe the corresponding generators
when $(X,Y) = (\R^{p,q},\R^{p-1,q})$ with $p+q$ $(=n) = 4$.

As we have seen in Theorems \ref{thm:dim} and \ref{thm:StageC},
there are two types of conditions on $(i,j)$, namely,
\begin{equation*}
-1 \leq i-j \leq 2 \quad \text{or} \quad n-2 \leq i+j \leq n+1,
\end{equation*}
for which nontrivial differential symmetry breaking operators
$\mathcal{E}^i(X)_u \to \mathcal{E}^j(Y)_v$ exist for some $u,v\in\C$.
(The latter inequality arises from the composition of the Hodge star
operator with respect to the pseudo-Riemannian metric.)
It turns out that \eqref{eqn:two} happens only if these two conditions 
are simultaneously fulfilled, that is, only if 
\begin{equation*}
-1 \leq i- j \leq n \quad \text{and} \quad n-2\leq i+j \leq n+1.
\end{equation*}

The four-dimensional case is illustrative to understand \eqref{eqn:two}
for the arbitrary dimension $n$. We give a complete list of parameters
$(i,j,u,v)$ for which \eqref{eqn:two} happens together with explicit generators
of $\Diff_{\conf(X;Y)}(\mathcal{E}^i(X)_u,\mathcal{E}^j(Y)_v)$.

Let $X=\R^{p,q}_+$ and $Y=\R^{p-1,q}_+$ with $n=p+q=4$.
We shall simply write as $(X,Y) = (\R^{p,q},\R^{p-1,q})$. 
(The case $(X,Y)=(\R^{p,q}_-,\R^{p,q-1}_-)$ is essentially the same and we omit it.)
We set
\begin{alignat*}{4}
&A:=\rest_{x_4=0} \circ\, &&d, \quad
&&B:=\rest_{x_4=0} \circ\, &&d^*,\\
&C:=\rest_{x_4=0} \circ\, &&\iota_{\frac{\partial}{\partial x_4}}d, \quad
&&D:=\rest_{x_4=0} \circ\, &&\iota_{\frac{\partial}{\partial x_4}}d^*.
\end{alignat*}
By using the formul{\ae} in \cite[Ch.~8.\ Sect.~5]{KKP16},
we readily see that
\begin{equation}\label{eqn:ABCD}
D \circ *_{\R^{p,q}}=\pm *_{\R^{p-1,q}} \circ A,\quad
C\circ *_{\R^{p,q}}=\pm*_{\R^{p-1,q}}\circ B.
\end{equation}

\begin{thmintro}\label{thm:F}
Suppose $(X,Y) = (\R^{p,q},\R^{p-1,q})$ with
$p+q =4$ and $p\geq 1$. Then \eqref{eqn:two} occurs 
if and only if $(i,j, u,v)$ 
appears in the nonempty boxes
in the table below.
Moreover, the pairs of operators in the table provide generators
of $\Diff_{\conf(X;Y)}(\mathcal{E}^i(X)_u,\mathcal{E}^j(Y)_v)$.

\footnotesize{
\begin{table}[H]
\begin{center}
\setlength\tabcolsep{1pt}.
\begin{tabular}{c|c|c|c|c}
\backslashbox{$i$}{$j$}& $0$ & $1$ & $2$ & $3$  \\
\hline
\; $0$ \; & & & & \\
\hline
&  & $u=0$, $v=1$, & 
        $u=v=0$, & \rule[0mm]{0mm}{5mm} \\
&  & $\ell=1$ & 
        $\ell=1$ & \rule[0mm]{0mm}{5mm} \\[5pt]
\cdashline{3-4} 
\; \raisebox{12pt}{$1$} \;&  & $*_{\R^{p-1,q}}\circ A$, & $A$, & \rule[0mm]{0mm}{5mm} \\
                                    &  
                                    &$C$
                                    &$*_{\R^{p-1,q}}\circ C$
                                    & \rule[0mm]{0mm}{3mm} \\[5pt]
\hline
&\; $u=0$, $v=3$, \; &\, $v-u\in \N_+$, & 
    \,$v-u\in\N$, &\, $u=v=0$, \rule[0mm]{0mm}{5mm} \\
&\,$\ell=1$ \, &\, $\ell=v-u-1$\, & 
   $\ell=v-u$\, &\,$\ell=1$ \rule[0mm]{0mm}{5mm} \\[5pt]
\cdashline{2-5} 
\; \raisebox{12pt}{$2$} \;
& $D$,  &$(\mathcal{D}^{2\to 1}_{u,\ell})_+$, & $(\mathcal{D}^{2\to 2}_{u,\ell})_+$,
& $A$, \rule[0mm]{0mm}{5mm} \\
                                  &$*_{\R^{p-1,q}}\circ A$
                                  &\; $*_{\R^{p-1,q}}\circ(\mathcal{D}^{2\to 2}_{u,\ell})_+$ \;
                                  &\; $*_{\R^{p-1,q}}\circ(\mathcal{D}^{2\to 1}_{u,\ell})_+$ \;
                                  &\; $*_{\R^{p-1,q}}\circ D$\, 
                                   \rule[0mm]{0mm}{5mm} \\[5pt]
\hline
& & $u=-2$, $v=1$, & 
        \, $u=-2$, $v=0$, \, & \rule[0mm]{0mm}{5mm} \\
& & $\ell=1$ & 
        \,  $\ell=1$ \, & \rule[0mm]{0mm}{5mm} \\[5pt]
\cdashline{3-4} 
\; \raisebox{12pt}{$3$} \;
&  & $D$, & $B$, & \rule[0mm]{0mm}{5mm} \\
                                  &  
                                  & $*_{\R^{p-1,q}}\circ B$ 
                                  & $*_{\R^{p-1,q}}\circ D$  
                                  & \rule[0mm]{0mm}{3mm} \\[5pt]
\hline
\; $4$ \;& & & &
\end{tabular}
\end{center}
\end{table}}%
\end{thmintro}

\begin{rem}
(1) By the duality theorem \cite[I, Thm.~9]{KP16},
the multiplicity in the branching laws of the 
generalized Verma modules, given as the dimension of \eqref{eqn:Verma} 
is also equal to 2 for 
the parameters in Theorem \ref{thm:F} (cf.\ \cite{K12}).\\
(2) For $(p,q)=(1,3)$, Maxwell's equations are expressed as
$d\alpha = 0$ and $d^*\alpha = 0$ for $\alpha \in \mathcal{E}^2(\R^{1,3})$,
see \cite{KW14} for instance.\\
(3) $\mathcal{D}^{2 \to 1}_{u,\ell}$ reduces to $-\rest_{x_4=0}\circ d^*$ if
$(u,\ell) = (0,1)$.
\end{rem}

\begin{proof}[Proof of Theorem \ref{thm:F}]
The assertions follow from \cite[Thms.~1.1 and 2.10]{KKP16}
owing to Theorems \ref{thm:161456} and \ref{thm:StageC}.
\end{proof}



\end{document}